\documentclass[12pt,a4paper]{amsart}
\usepackage{graphics}
\usepackage{amsfonts}
\usepackage{amsmath}
\usepackage{amsthm}
\usepackage{amssymb}
\usepackage{physics}
\usepackage[dvipsnames]{xcolor}
\usepackage{pdfpages}
\usepackage{placeins}
\usepackage{tkz-euclide}
\usepackage{setspace}
\usepackage[utf8]{inputenc}
\usepackage[margin=0.01cm]{geometry}
\usepackage{anyfontsize}
\usepackage{faktor}
\usepackage{todonotes}
\usepackage{caption}
\newcommand{\ga}{\alpha}
\newcommand{\gb}{\beta}
\newcommand{\gc}{\gamma}
\newcommand{\gd}{\delta}
\newcommand{\bigslant}[2]{{\raisebox{.4em}{$#1$}\Big/\raisebox{-.2em}{$#2$}}}
\newtheorem*{thm:maintheorem}{Theorem \ref{thm:maintheorem}}
\usetikzlibrary{3d}
\makeatletter
\usetikzlibrary{decorations.pathreplacing,calligraphy}
\usepackage [english]{babel}
\usepackage [autostyle, english = american]{csquotes}
\MakeOuterQuote{"}
\newtheorem{thm}{Theorem}[section]
\newtheorem{defin}[thm]{Definition}
\newtheorem{prop}[thm]{Proposition}

\newtheorem{lem}[thm]{Lemma}

\usepackage{tikz,calc}
\usepackage{tikzscale}
\usetikzlibrary{calc}
\DeclareMathOperator{\sH}{H}
\DeclareMathOperator{\sP}{P}
\usepackage{fp}
\usetikzlibrary{trees,arrows,positioning,shapes.geometric,angles,calc,patterns,quotes}
\usetikzlibrary{decorations.pathmorphing}
\title{Characterising Trees and Hyperbolic Spaces by their Boundaries}

\author{Isobel Davies}
\address{Otto-von-Guericke Universit\"at Magdeburg, Universit\"atsplatz 2, 39106 Magdeburg}
\email{isobel.davies@ovgu.de}

\date{}

\begin{document}
 \captionsetup[figure]{labelfont={bf},name={Fig.},labelsep=space}
\makeatletter
\def\blfootnote{\xdef\@thefnmark{}\@footnotetext}
\makeatother
\blfootnote{\textbf{MSC2020 Subject Classification:} 05C05, 51M10.}
\blfootnote{\textbf{Keywords:} trees, hyperbolic spaces, cross ratio, proper CAT(-1) spaces, visual boundary.}
\begin{abstract}
We use the language of proper CAT(-1) spaces to study thick, locally compact trees, the real, complex and quaternionic hyperbolic spaces and the hyperbolic plane over the octonions. These are rank 1 Euclidean buildings, respectively rank 1 symmetric spaces of non-compact type. We give a uniform proof that these spaces may be reconstructed using the cross ratio on their visual boundary, bringing together the work of Tits and Bourdon.
\end{abstract}
\maketitle
\section{Introduction}
Euclidean buildings and symmetric spaces of non-compact type are two important families of metric spaces of non-positive curvature, with a number of common properties. For example, their maximal flat subspaces are all of the same dimension (this is the rank) and their boundary points admit the structure of a spherical building. We study the rank 1 case. The $n$-dimensional \textit{hyperbolic spaces} $\mathbb{K}\sH^n$ over $\mathbb{K}=\mathbb{R},\mathbb{C},\mathbb{H}$ ($n\geq 2$) and $\mathbb{O}$ ($n=2$) are defined in Section~\ref{Section:HyperbolicSpaces}. These are the rank 1 symmetric spaces of non-compact type. A (thick, locally compact, discrete) rank 1 Euclidean building is the metric realisation of a connected acyclic graph whose vertices have finite degree at least 3, we therefore refer to these spaces as \textit{trees}, see Section~\ref{Section:Trees}.

Tits gives a characterisation of trees $T$ in terms of their boundary points $\partial T$ and canonical valuation $\omega_T:(\partial T)^4\rightarrow\mathbb{R}$ in \cite{tits}. He then uses this to prove that all Euclidean buildings can be characterised by their boundary. Having already classified spherical buildings of dimension $\geq 2$ this allowed him to give a classification for Euclidean buildings of rank $\geq 3$. A detailed account of Tits' work is given in \cite{weiss}. 

Bourdon defines and studies the cross ratio $\omega$ on the boundary of a proper CAT(-1) space in \cite{bourdon96}. This simultaneously generalises the projective cross ratio on the boundary of the real hyperbolic plane and the canonical valuation defined by Tits for trees. Bourdon proves that a hyperbolic space $H$ is isometrically embedded in a proper CAT(-1) space $X$ if and only if there is a cross-ratio preserving topological embedding of $\partial H$ into $\partial X$. In particular, this characterises when two hyperbolic spaces are isometric, in terms of their boundary points and the cross ratio.

Our main result is a uniform reconstruction of trees and hyperbolic spaces, using their boundary points and the cross ratio. In Tits' work, he shows that a tree may be reconstructed using its boundary data. However, to the best of the author's knowledge, this was not known for hyperbolic spaces.

Before stating the main theorem, we briefly explain some notation. Tits parametrises the vertices of a tree $T$ using triples of distinct boundary points $\ga,\gb,\gc\in\partial T$, the corresponding point is the centre of the ideal tripod $\Delta(\ga,\gb,\gc)$. Bourdon generalises this, defining the point $p(\ga,\gb;\gc)$ in a proper CAT(-1) space $X$ for $\ga,\gb,\gc\in\partial X$, in terms of the Bourdon metric (see Definition~\ref{Def:SpecialPoint}). Let $[\ga,\gb]_\gc:\mathbb{R}\rightarrow [\ga,\gb]\subset X$ be the unique parametrisation of the geodesic line from $\ga$ to $\gb$ such that $[\ga,\gb]_{\gc}(0)=p(\ga,\gb;\gc)$. We define a model $\chi_{\ga\gb}$ for $[\ga,\gb]\subset X$, letting $(\ga,\gb,\gc,t)\in\chi_{\ga\gb}$ correspond to $[\ga,\gb]_{\gc}(t)$ (see Proposition~\ref{Prop:geodesicmetric}). 

Suppose that $X$ is a hyperbolic space or a tree. In Definition~\ref{Def:equivalencerelation}, we define an equivalence relation on the disjoint union of $\chi_{\ga\gb}$ over all distinct $\ga,\gb\in\partial X$, which models how the geodesics in $X$ intersect. The set of equivalence classes is denoted $\Omega(X)$ and is in bijection with $X$. Since any two points in $X$ are contained in a unique geodesic $[\ga,\gb]$, which may be parametrised by $[\ga,\gb]_{\gc}$ for any $\gc\in\partial X\setminus\{\ga,\gb\}$, we can write any two elements of $\Omega(X)$ as $(\ga,\gb,\gc,s)$ and $(\ga,\gb,\gc,t)$ for some $\ga,\gb,\gc\in\partial X$ and some $s,t\in\mathbb{R}$. The definition of $\Omega(X)$ uses only the boundary points and cross ratio. The following theorem therefore shows that one may construct a metric space $(\Omega(X),d_{\omega})$ isometric to $(X,d_X)$, using only the boundary data $(\partial X,\omega)$. 

\begin{thm}\label{thm:maintheorem}
Let $(X,d_X)$ be either a hyperbolic space or a tree. 
Then, the map
\begin{align*}
&d_{\omega}:\Omega(X)\times\Omega(X)\rightarrow\mathbb{R}\qquad
&d_{\omega}((\ga,\gb,\gc,s),(\ga,\gb,\gc,t)):=\abs{s-t}
\end{align*}
is a well-defined metric on $\Omega(X)$ and 
\begin{align*}
\varphi:\Omega(X)\rightarrow X\qquad
\varphi:(\ga,\gb,\gc,t)\mapsto[\ga,\gb]_{\gc}(t)
\end{align*}
is an isometry.
\end{thm}

In Section~\ref{Section:Preliminaries} we outline the necessary preliminaries for this paper, in particular we define hyperbolic spaces, trees and proper CAT(-1) spaces, recalling some of their properties. We also define the Gromov product and Bourdon metric on the visual boundary of a proper CAT(-1) space. 

In Section~\ref{Section:ProperCAT(-1)} we recall the definition of the cross ratio and outline what is known about its relationship with the geometry of proper CAT(-1) spaces. In particular, we define $\chi_{\ga\gb}$ and prove that this models $[\ga,\gb]$. We also recall the Ptolemy inequality, proven by Foertsch and Schroeder \cite{foertschschroeder} which gives a condition in terms of the cross ratio for when an ideal quadrilateral in a proper CAT(-1) space is isometric to an ideal quadrilateral in the real hyperbolic plane.

In Section~\ref{Section:Reconstruction} we prove a number of uniform propositions for trees and hyperbolic spaces, showing that one may use the cross ratio to understand how and when any two geodesics intersect. We then use this to define the equivalence relation on the disjoint union of all $\chi_{\ga\gb}$ and prove Theorem~\ref{thm:maintheorem}.

\section{Preliminaries}\label{Section:Preliminaries}
\subsection{Geodesics and The Visual Boundary} Let $(X,d)$ be a metric space. Let $I\subseteq \mathbb{R}$ be a closed interval. A curve $c:I\rightarrow X$ is called a \textit{geodesic} if $d(c(s),c(t))=\abs{s-t}$ for all $s,t\in I$. If $I=[a,b]$ then we call $c$ a \textit{geodesic segment} from $x=c(a)$ to $y=c(b)$ and denote its image $[x,y]$. If there exists a geodesic segment from $x$ to $y$ for all $x,y\in X$, then we say that $X$ is a \textit{geodesic metric space}. If $I=[0,\infty)$, then we call $c$ a \textit{geodesic ray} and if $I=\mathbb{R}$, then we call $c$ a \textit{geodesic line}.\\ The visual boundary of $X$ is defined to be the set of equivalence classes of geodesic rays
\begin{align*}
\partial X:=\{c:[0,\infty)\rightarrow X\mid d(c(s),c(t))=\abs{s-t}\}\mathord/\sim
\end{align*}
where $c\sim c'$ if and only if there exists $K>0$ such that $d(c(t),c'(t))<K$ for all $t\geq 0$. We will denote boundary points using greek letters and if a geodesic ray $c$ is in the equivalence class $\ga$, then we will write $c(\infty)=\ga$. For a more detailed introduction to geodesic metric spaces and the visual boundary, see \cite{bridsonhaefliger}.\\

\subsection{Symmetric Spaces of Non-Compact Type}
A \textit{symmetric space} is a connected Riemannian manifold $M$ such that for all $o\in M$, there exists an isometry $\varphi_o:M\rightarrow M$, fixing $o$ and reversing geodesics through $o$. It follows from the definition, that a symmetric space is geodesically complete and its isometry group acts transitively. If $M$ is a simply connected symmetric space with non-positive sectional curvature and no Euclidean de Rham factor, then it is said to be of \textit{non-compact type}.
 The \textit{rank} of a symmetric space of non-compact type is equal to the dimension of its maximal flat subspaces. i.e. subspaces which are isometric to Euclidean space, ordered by inclusion. The rank 1 symmetric spaces of non-compact type are the hyperbolic spaces defined in the next section. For more details on symmetric spaces of non-compact type, see \cite{eberlein}.

\subsection{The Hyperbolic Spaces}\label{Section:HyperbolicSpaces} Let $n\geq 2$ and let $\mathbb{K}$ denote either the field of real numbers $\mathbb{R}$, complex numbers $\mathbb{C}$ or the skew field of quaternions $\mathbb{H}$. Then one may define the following Hermitian form on  $\mathbb{K}^{n+1}$:
\begin{align}\label{Hermitian form}
\langle x\mid y\rangle :=\sum_{k=1}^n{\overline{x_k}y_k-\overline{x_{n+1}}y_{n+1}},    
\end{align}
where $x=(x_1,\ldots,x_{n+1})$ and $y=(y_1,\ldots,y_{n+1})$.  We define an equivalence relation on $\mathbb{K}^{n+1}\setminus\{0\}$ by $x\sim_{\mathbb{K}} y$ if and only if $x=\lambda y$ for some $\lambda\in\mathbb{K}$, the set of equivalence classes under this relation gives the $n$-dimensional \textit{projective space} $\mathbb{K}\sP^n$. The points of the $n$-dimensional \textit{hyperbolic space} $\mathbb{K}\sH^n$ are then defined to be the set of points $[x]\in \mathbb{K}\sP^n$ satisfying $\langle x\mid x\rangle<0$. The distance $d_H([x],[y])$ between two points of $\mathbb{K}\sH^n$ is defined to be the unique positive real number satisfying  
\begin{align}\label{hyperbolicmetric}
\cosh^2 d_H([x],[y])=\dfrac{\langle x\mid y\rangle\langle y\mid x\rangle}{\langle x\mid x\rangle\langle y\mid y\rangle},
\end{align}
where $x,y\in\mathbb{K}^{n+1}$ are any representatives of $[x],[y]\in\mathbb{K}\sH^n$. The above model is explained in detail in \cite[Chapter II.10]{bridsonhaefliger}.

Let $\mathbb{O}$ denote the division algebra of octonions and define $\langle x\mid y\rangle\in\mathbb{O}$ for $x,y\in\mathbb{O}^3$ as in (\ref{Hermitian form}), note that since $\mathbb{O}$ is not associative, this is not an Hermitian form. 
Denote by $\mathbb{O}^3_0$ the set of triples $x=(x_1,x_2,x_3)\in\mathbb{O}^3$ such that the subalgebra of $\mathbb{O}$ generated by $x_1,x_2,x_3$ is associative. 
We define an equivalence relation on $\mathbb{O}^3_0\setminus\{0\}$ by $x\sim_{\mathbb{O}} y$ if and only if $x=\lambda y$ and $\lambda,y_1,y_2,y_3$ are contained in an associative subalgebra of $\mathbb{O}$, in this case
$\langle  x\mid  x\rangle =\Bar{\lambda}\lambda\langle y\mid y\rangle\in\mathbb{R}$.    
The set of equivalence classes under $\sim_{\mathbb{O}}$ is the \textit{projective plane} $\mathbb{O}\sP^2$. The \textit{hyperbolic plane} $\mathbb{O}\sH^2$ is defined to be the set of points $[x]\in \mathbb{O}\sP^2$ satisfying $\langle x\mid x\rangle<0$. 
The distance $d_H([x],[y])$ between two points of $\mathbb{O}\sH^2$ is defined to be the unique positive real number satisfying (\ref{hyperbolicmetric}) where $x,y\in\mathbb{O}^3_0$ are any representatives of $[x],[y]$ with the property that for some $i\neq j$, $x_i$ and $y_j$ are real (cf. \cite{allcockop2} and \cite[Section 19]{mostow}).
The above model is based on \cite{allcockoh2}.

\subsection{Isometric Embeddings $\mathbb{R}\sH^2\hookrightarrow \mathbb{K}\sH^n$}\label{Section:IsometricEmbeddings}
 Since $\mathbb{R}^3\subset\mathbb{C}^3\subset\mathbb{H}^3\subset\mathbb{O}^3_0$, it is clear that there is a canonical isometric embedding of the real hyperbolic plane in the other hyperbolic spaces. In this section we look at other isometric embeddings of $\mathbb{R}\sH^2\hookrightarrow \mathbb{K}\sH^n$, where
$\mathbb{K}=\mathbb{R},\mathbb{C},\mathbb{H}$ $(n\geq 2)$ or $\mathbb{K}=\mathbb{O}$ ($n=2$), using the ideas in \cite[Theorem II.10.16]{bridsonhaefliger}. Let $[x]\in \mathbb{K}\sH^n$, since the isometry group acts transitively we may assume without loss of generality that $x=\begin{pmatrix}
    0,\ldots, 0, 1
\end{pmatrix}$ and we can model the tangent space at $[x]$ by $x^{\perp}$, where $x^{\perp}:=\{u\in \mathbb{K}^{n+1}\mid \langle x\mid u\rangle=0\}$ for $\mathbb{K}=\mathbb{R},\mathbb{C},\mathbb{H}$ and $x^{\perp}:=\{u\in \mathbb{O}^{3}_0\mid \langle x\mid u\rangle=0\}$ for $\mathbb{K}=\mathbb{O}$. For all $\mathbb{K}$, $x^{\perp}=\mathbb{K}\times\ldots\times\mathbb{K}\times\{0\}\cong \mathbb{K}^n$. 
The geodesics with $c(0)=[x]$ are given by $c(t)=[\cosh(t)\cdot x+\sinh(t)\cdot u]$, where $u$ is a unit tangent vector i.e. $u\in x^{\perp}$ with $\langle u\mid u\rangle=1$. For all $u\in x^{\perp}\setminus\{0\}$, $\langle u\mid u\rangle> 0$ (positive definiteness) and the real part of $\langle u\mid v\rangle$ is equal to the Riemannian metric for all $u,v\in x^{\perp}$. 

Let $u,v\in x^{\perp}$ be distinct unit tangent vectors and suppose that $\langle u\mid v\rangle\in \mathbb{R}$, then the real span of $u,v,x$ is a 3-dimensional vector subspace $V\subset\mathbb{K}\times\ldots\times\mathbb{K}\times\mathbb{R}$ with the property that for all $x,y\in V$, $x\sim_\mathbb{K} y$ if and only if $x\sim_\mathbb{R}y$. By positive definiteness we may find $\tilde{u},\tilde{v}\in V\cap x^{\perp}$ such that $\langle \tilde{u}\mid \tilde{u}\rangle =1$, $\langle \tilde{v}\mid \tilde{v}\rangle =1$ and $\langle \tilde{u}\mid \tilde{v}\rangle =0$. It follows that there is a linear isometry (i.e. a bijective $\langle\cdot\mid\cdot\rangle$ - preserving map) sending $\mathbb{R}^3$ onto $V$. The image of $V\setminus\{0\}$ in $\mathbb{K}\sP^n$ is isomorphic to $\mathbb{R}\sP^2$ and its image in $\mathbb{K}\sH^n$ is isometric to $\mathbb{R}\sH^2$. 

\subsection{Geodesic Triangles and The Gromov Product}\label{Section:Gromov Product}
Let $(X,d)$ be a metric space. For $o,x,y\in X$ the \textit{Gromov product} of $x$ and $y$ based at $o$ is defined to be
\begin{align*}
(x\mid y)_{o}:=\tfrac{1}{2}(d(x,o)+d(y,o)-d(x,y)).    
\end{align*}
By the triangle inequality, the Gromov product $(x\mid y)_o$ is a non-negative real number for all $x,y,o\in X$.
Suppose that for $x,y,z\in X$, there exist geodesic segments $[x,y],[x,z],[y,z]$, then a \textit{geodesic triangle} of $x,y,z$ is defined to be $\Delta(x,y,z):=[x,y]\cup[x,z]\cup[y,z]$. In a geodesic metric space, the Gromov product can be understood in terms of points in a triangle, see Figure~\ref{Fig:GromovProduct}. All hyperbolic triangles $\Delta(x,y,z)\subset \mathbb{K}\sH^n$ have the property that if $p\in[x,y]$, $q\in[x,z]$ such that $d(x,p)=d(x,q)\leq (y\mid z)_x$, then $d(p,q)\leq 2\ln\varphi$, where $\varphi$ denotes the golden ratio. Inspired by this property, a geodesic metric space $X$ is called \textit{$\gd$-hyperbolic} for some $\gd\geq 0$, if all geodesic triangles $\Delta(x,y,z)\subset X$ have the property that if $p\in[x,y]$, $q\in[x,z]$ such that $d(x,p)=d(x,q)\leq (y\mid z)_x$, then $d(p,q)\leq \gd$, see \cite{buyaloschroeder} for more details on $\gd$-hyperbolic spaces. 
\begin{figure}
    \centering
\includegraphics[width=0.8\linewidth]{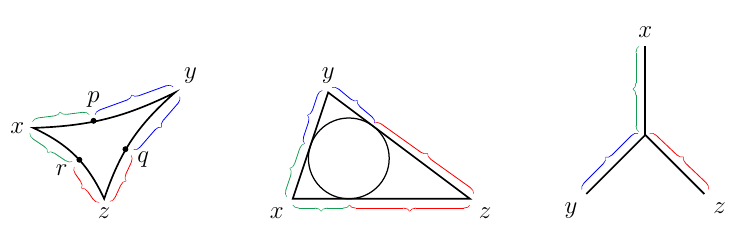}
    \caption{Let $(X,d)$ be a geodesic metric space and consider a geodesic triangle $\Delta(x,y,z)$. There exist three unique points  $p,q,r\in\Delta(x,y,z)$, such that \textcolor{ForestGreen}{$d(x,p)=d(x,r)$}, \textcolor{blue}{$d(y,p)=d(y,q)$} and \textcolor{red}{$d(z,q)=d(z,r)$}. The left picture of this figure depicts an arbitrary geodesic triangle with these three special points marked and their respective distances to the vertices highlighted. These distances are equal to the Gromov products \textcolor{ForestGreen}{$(y\mid z)_x$}, \textcolor{blue}{$(x\mid z)_y$} and \textcolor{red}{$(x\mid y)_z$}.  If $\Delta(x,y,z)$ is a Euclidean triangle, these three special points are exactly the three intersection points of $\Delta(x,y,z)$ with its incircle, this is pictured in the middle. The right picture depicts a tripod. Here, the three special points meet at the centre.}
    \label{Fig:GromovProduct}
\end{figure}

\subsection{Trees} \label{Section:Trees}
A building is a simplicial complex $\Delta$ which can be expressed as a union of Coxeter complexes (called apartments) such that any two simplices $A,B\in\Delta$ are contained in an apartment and for any two apartments containing $A,B\in\Delta$, there exists an isomorphism between them, fixing their intersection pointwise, see \cite{abramenkobrown} for a detailed account of buildings. We always require that our buildings are thick, locally finite simplicial complexes i.e. every codimension 1 simplex in $\Delta$ is the face of $n$ maximal simplices where $3\leq n<\infty$. A building is called Euclidean if its apartments are Euclidean Coxeter complexes.
Every Euclidean building may be realised as a complete metric space of non-positive curvature \cite[Section 11.2.]{abramenkobrown}. The \textit{rank} of a Euclidean building is defined to be the dimension of its maximal flat subspaces (these correspond to apartments). A rank 1 Euclidean building is the metric realisation of a connected acyclic graph whose vertices have finite degree at least 3. We will refer to these spaces simply as \textit{trees}. Note that all trees are $0$-hyperbolic since all their geodesic triangles are tripods.
\subsection{Proper CAT(-1) Spaces}
Let $(X,d)$ be a geodesic metric space and consider a geodesic triangle $\Delta(x,y,z)\subset X$. A geodesic triangle $\Bar{\Delta}(\Bar{x},\Bar{y},\Bar{z})\subset \mathbb{R}\sH^2$ is called a \textit{comparison triangle} for $\Delta(x,y,z)\subset X$ if and only if $d(x,y)=d(\Bar{x},\Bar{y}),$ $d(x,z)=d(\Bar{x},\Bar{z})$ and $d(y,z)=d(\Bar{y},\Bar{z})$.

\begin{defin}
A geodesic metric space $X$ is called \textit{CAT(-1)} if for all geodesic triangles $\Delta(x,y,z)\subset X$ and all $p\in[x,y]\subset\Delta,\,q\in[x,z]\subset\Delta$, 
\begin{align*}
d(p,q)\leq d(\Bar{p},\Bar{q}),    
\end{align*}
 where $\Bar{p}\in[\Bar{x},\Bar{y}]\subset\Bar{\Delta}$ and $\Bar{q}\in[\Bar{x},\Bar{z}]\subset\Bar{\Delta}$ are comparison points (i.e. $d(x,p)=d(\Bar{x},\Bar{p}),$ $d(x,q)=d(\Bar{x},\Bar{q})$) in the comparison triangle $\Bar{\Delta}(\Bar{x},\Bar{y},\Bar{z})\subset \mathbb{R}\sH^2$.    
\end{defin}

The hyperbolic spaces are all CAT(-1) spaces \cite[Theorem II.10.10.]{bridsonhaefliger} as are all trees \cite[II.1.15]{bridsonhaefliger}.
A metric space is called \textit{proper} if every closed and bounded subset is compact. Since trees and hyperbolic spaces are locally compact, complete, geodesic metric spaces, it follows from the Hopf-Rinow theorem \cite[Proposition I.3.7]{bridsonhaefliger} that all trees and hyperbolic spaces are proper.
\subsection{Geodesics in CAT(-1) Spaces} 
Let $X$ be a CAT(-1) space. For all $x,y\in X$, there exists a geodesic segment from $x$ to $y$, which is unique up to parametrisation \cite[Proposition II.1.4]{bridsonhaefliger}.
If $X$ is complete, then for all $\ga\in\partial X$ and all $o\in X$, there exists a unique geodesic ray $c$ such that $c(0)=o,\,c(\infty)=\ga$  \cite[Proposition II.8.2]{bridsonhaefliger}. We denote the image of this geodesic ray $[o,\ga]$. If $X$ is proper, then for any two distinct boundary points $\ga,\gb\in\partial X$, there exists a geodesic line $c$ such that $c(-\infty)=\ga$ and $c(\infty)=\gb$. We denote the image of this geodesic by $[\ga,\gb]$. It follows from the flat strip theorem \cite[Theorem II.2.13]{bridsonhaefliger} that this geodesic is unique up to parametrisation.
\subsection{The Gromov Product and Bourdon Metric on $\partial X$}\label{Section:BourdonMetric}
Let $X$ be a proper CAT(-1) space, then we define the \textit{Gromov product} of $\ga,\gb\in\partial X$ based at $o\in X$ as \begin{align*}
(\ga\mid\gb)_{o}:=\lim_{t\rightarrow\infty}(x(t)\mid y(t))_{o}
\end{align*}
where $x\in\ga$, $y\in\gb$ \cite[Proposition 3.4.2]{buyaloschroeder}.

Define the \textit{horospherical distance} between $x,y\in X$ relative to $\ga\in\partial X$ as
\begin{align*}
B_\ga(x,y):=\lim_{t\rightarrow\infty}\left(d(x,r(t))-d(y,r(t))\right),
\end{align*}
where $r\in\ga$. 

We can understand a base change of the Gromov product in terms of the horospherical distance \cite[2.4.2]{bourdon95} as follows   
\begin{align*}(\ga\mid\gb)_y=(\ga\mid\gb)_x-\tfrac{1}{2}\left(B_{\ga}(x,y)+B_{\gb}(x,y)\right), \qquad \forall x,y\in X,\,\ga,\gb\in\partial X. \end{align*} 
If $X$ is the real hyperbolic plane then the Gromov product may be understood in terms of angles. 
Let $o\in\mathbb{R}\sH^2$ and $\ga,\gb\in\partial \mathbb{R}\sH^2$, and let $\theta_o(\ga,\gb)$ denote the angle between the geodesic rays $[o,\ga]$ and $[o,\gb]$, then
\begin{align}\label{Bourdonangle}
e^{-(\ga\mid\gb)_o}=\sin{\dfrac{\theta_o(\ga,\gb)}{2}}. 
\end{align}
This property follows directly from the hyperbolic cosine rule \cite[Section 2.4.3.]{buyaloschroeder}.
Using the Gromov product, Bourdon defines \begin{align*}
d_o(\ga,\gb):=\begin{cases}
 e^{-(\ga\mid\gb)_o}&\text{ if }\ga\neq\gb\\
 0&\text{ if }\ga=\gb
\end{cases}
\end{align*}
and uses (\ref{Bourdonangle}) to show that this gives rise to a metric on $\partial X$ \cite[Theorem 2.5.1]{bourdon95}. 
The Bourdon metric on $\partial X$ takes values in $[0,1]$ where $d_o(\ga,\gb)=1$ if and only if $o\in[\ga,\gb]$. 

\section{The Cross Ratio on the Boundary of a Proper CAT(-1) Space}\label{Section:ProperCAT(-1)}
In this section $X$ denotes a proper CAT(-1) space and $\partial X$ its visual boundary. Suppose throughout that $\abs{\partial X}\geq 3$.  
\begin{defin}\label{Def:SpecialPoint}\cite{bourdon96}
Let $\ga,\gb,\gc\in\partial X$ be three distinct boundary points. The point $p(\ga,\gb;\gc)\in X$ is defined to be the unique point $p\in[\ga,\gb]$ such that $d_p(\ga,\gc)=d_p(\gb,\gc)$.
\end{defin}

The existence and uniqueness of $p(\ga,\gb;\gc)$ can easily be verified using the relationship between the Gromov product and the horospherical distance outlined in Section~\ref{Section:BourdonMetric}.
It also follows from this relationship that the cross ratio, defined below, does not depend on the choice of basepoint $o\in X$.

\begin{figure}
    \centering
\includegraphics[width=0.6\linewidth]{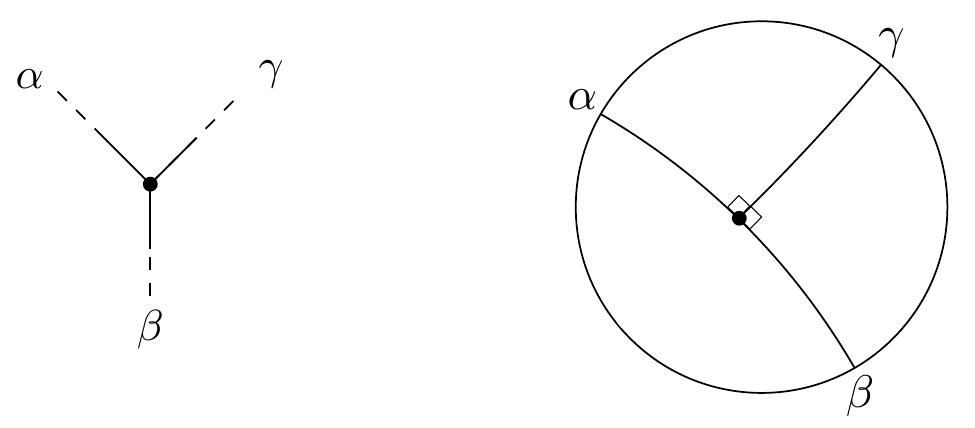}
    \caption{If $X$ is a tree, then the point $p(\ga,\gb;\gc)$ is the centre of the tripod $\Delta(\ga,\gb,\gc)$.  
To see this, note that for $p\in[\ga,\gb]\cap[\ga,\gc]\cap[\gb,\gc]$, we have $d_p(\ga,\gb)=d_p(\ga,\gc)=d_p(\gb,\gc)=1$. 
If $X$ is the real hyperbolic plane, then it follows from (\ref{Bourdonangle}) that $p(\ga,\gb;\gc)$ is the unique point $p\in[\ga,\gb]$, such that $[p,\gc]$ is perpendicular to $[\ga,\gb]$.}
\label{fig:specialpoint}
\end{figure}

\begin{defin}\label{Def:crossratio}\cite{bourdon96}
The cross ratio of pairwise distinct $\ga,\gb,\gc,\gd\in\partial X$ is defined as follows
\begin{align*}
\omega(\ga,\gb;\gc,\gd):=\frac{d_o(\ga,\gc)d_o(\gb,\gd)}{d_o(\ga,\gd)d_o(\gb,\gc)},
\end{align*}
where $o\in X$ is any basepoint.
\end{defin}

\begin{prop}\cite[Proposition 1.3]{bourdon96}\label{Prop:crossratiometric}
Let $\ga,\gb,\gc,\gd\in\partial X$ be pairwise distinct. Then,
\begin{align*}
\ln{\omega(\ga,\gb;\gc,\gd)}=
\begin{cases}
-d(p(\ga,\gb;\gc),p(\ga,\gb;\gd))&\text{ if }p(\ga,\gb;\gc)\in[p(\ga,\gb;\gd),\ga]\\
d(p(\ga,\gb;\gc),p(\ga,\gb;\gd))&\text{ if }p(\ga,\gb;\gc)\in[p(\ga,\gb;\gd),\gb].
\end{cases}
\end{align*}
\end{prop}

\begin{defin}\label{Def:parametrisedgeodesicnotation}
Let $\ga,\gb,\gc\in\partial X$ be pairwise distinct. Denote by $[\ga,\gb]_{\gc}:\mathbb{R}\rightarrow X$ the parametrisation of $[\ga,\gb]$ satisfying
\begin{align*}
[\ga,\gb]_{\gc}(-\infty)=\ga,\qquad 
[\ga,\gb]_{\gc}(0)=p(\ga,\gb;\gc), \qquad
[\ga,\gb]_{\gc}(\infty)=\gb.
\end{align*}
\end{defin}
\begin{lem}\label{Lem:parametriseddistance}
Let $\ga,\gb,\gc,\gd\in\partial X$ be pairwise distinct. Given two different parametrisations of $[\ga,\gb]$, one can use the cross ratio to determine whether they are describing the same point:
\begin{align*}
[\ga,\gb]_{\gc}(t)=[\ga,\gb]_{\gd}(s)\iff \ln \omega(\ga,\gb;\gc,\gd)=s-t
\end{align*}
\end{lem}
\begin{proof}
Without loss of generality, suppose that $\ln\omega(\ga,\gb;\gc,\gd)\geq 0$, then by Proposition~\ref{Prop:crossratiometric} we know that
$\ln\omega(\ga,\gb;\gc,\gd)=d(p(\ga,\gb;\gc),p(\ga,\gb;\gd))$
and $p(\ga,\gb;\gc)\in[p(\ga,\gb;\gd),\gb]$. It follows that
$[\ga,\gb]_{\gc}(t)=[\ga,\gb]_{\gd}(t+\ln\omega(\ga,\gb;\gc,\gd))$.
\end{proof}
Both Bourdon and Tits used boundary points in order to model the geodesics of their spaces.
Below, we define $\chi_{\ga\gb}$ for $\ga,\gb\in\partial X$ and prove that this models the geodesic $[\ga,\gb]\subset X$.
\begin{prop}\label{Prop:geodesicmetric}
For distinct $\ga,\gb\in\partial X$, define
\begin{align*}
\chi_{\ga\gb}:=\{(\ga,\gb,\gc,t)\mid \gc\in\partial X\setminus\{\ga,\gb\}, t\in\mathbb{R}\}\mathord/\sim
\end{align*}
where $(\ga,\gb,\gc,t)\sim(\ga,\gb,\gd,s)$ if and only if $\ln\omega(\ga,\gb;\gc,\gd)=s-t$. Then,
\begin{align*}
d_{\omega}:\chi_{\ga\gb}\times\chi_{\ga\gb}\rightarrow\mathbb{R}\qquad
d_{\omega}((\ga,\gb,\gc,t),(\ga,\gb,\gc,s)):=\abs{s-t}
\end{align*}
is a well-defined metric on $\chi_{\ga\gb}$ and
\begin{align*}
\varphi_{\ga\gb}:\chi_{\ga\gb}\rightarrow [\ga,\gb]\qquad \varphi:(\ga,\gb,\gc,t)\mapsto [\ga,\gb]_{\gc}(t)
\end{align*}
is an isometry.
\end{prop}
\begin{proof}
 Fix pairwise distinct $\ga,\gb,\gc\in\partial X$. Every point of $[\ga,\gb]$ can be expressed as $[\ga,\gb]_{\gc}(t)$ for some $t\in\mathbb{R}$. It follows from Lemma~\ref{Lem:parametriseddistance} that $\varphi_{\ga\gb}$ is a bijection and therefore every point of $\chi_{\ga\gb}$ can be expressed as $(\ga,\gb,\gc,t)$ for some $t\in\mathbb{R}$ and $d_{\omega}$ is well-defined. Since $[\ga,\gb]_{\gc}$ is a geodesic we know that
\begin{align*}
d([\ga,\gb]_{\gc}(s),[\ga,\gb]_{\gc}(t))=\abs{s-t}
\end{align*} and 
therefore
\begin{align*}
d_{\omega}((\ga,\gb,\gc,s),(\ga,\gb,\gc,t))=d(\varphi_{\ga\gb}(\ga,\gb,\gc,s),\varphi_{\ga\gb}(\ga,\gb,\gc,t))
\end{align*}
which shows that $d_{\omega}$ is a metric and $\varphi_{\ga\gb}$ is an isometry.
\end{proof}
Foertsch and Schroeder study CAT(-1) spaces and Gromov hyperbolic spaces in \cite{foertschschroeder}, proving the `Ptolemy inequality' for complete CAT(-1) spaces, in terms of the Bourdon metric on the Gromov boundary. We present this theorem for proper CAT(-1) spaces, in terms of the cross ratio on the visual boundary.
\begin{thm}\label{Thm:hyperbolicintersects}\cite[Theorem 1.1]{foertschschroeder}
Let $X$ be a proper CAT(-1) space. For any pairwise distinct $\ga,\gb,\gc,\gd\in\partial X$,
\begin{align*}
\omega(\ga,\gd;\gc,\gb)+\omega(\ga,\gc;\gd,\gb)\geq 1.  
\end{align*}
Equality holds if and only if $[\ga,\gb]$ and $[\gc,\gd]$ intersect and the convex hull of $\ga,\gb,\gc,\gd$ is isometric to an ideal quadrilaterial in $\mathbb{R}\sH^2$.
\end{thm}

It follows from the above theroem, that we can use the cross ratio in order to characterise subspaces of the hyperbolic spaces which are isometric to the real hyperbolic plane. We also have a characterisation of all intersecting geodesics in the real hyperbolic plane in terms of the cross ratio. In the next section, we will build on this, giving a characterisation of all intersecting geodesics in all trees and hyperbolic spaces. In order to do this, we make use of a property of intersecting geodesics in the hyperbolic spaces, outlined by Bourdon in \cite{bourdon96} in his proof of Theorem 0.1. 
\section{Uniform Reconstruction of $X$}\label{Section:Reconstruction}
In this section $X$ denotes either a hyperbolic space or a tree (see Sections~\ref{Section:HyperbolicSpaces} and \ref{Section:Trees}). We use the cross ratio in order to understand how and when two geodesics in $X$ intersect. Finally, we use this to construct a metric space $\Omega(X)$ which is isometric to $X$.

It follows from Theorem~\ref{Thm:hyperbolicintersects} that for a real hyperbolic space, the geodesics $[\ga,\gb]$ and $[\gc,\gd]$ intersect if and only if $\omega(\ga,\gc;\gd,\gb)+\omega(\ga,\gd;\gc,\gb)=1$. The first step is to find an analogue of this property for trees.

\begin{figure}
    \centering
    \includegraphics[width=0.9\linewidth]{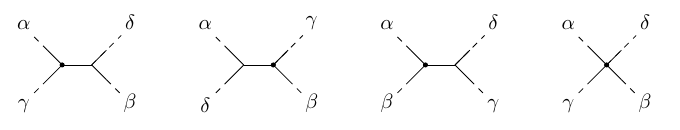}
    \caption{Let $X$ be a tree, then for distinct  $\ga,\gb,\gc,\gd\in\partial X$, there are only four possibilities for how the six geodesics between them  intersect. In each case, the point $p(\ga,\gb;\gc)$ is marked.}
    \label{fig:trees}
\end{figure}

\begin{lem}\label{Lem:fourcases}
Let $X$ be a tree. Given distinct $\ga,\gb,\gc,\gd\in\partial X$, there are only four possible cases for the values of $\omega(\ga,\gb;\gc,\gd)$, $\omega(\ga,\gc;\gd,\gb)$  and $\omega(\ga,\gd;\gc,\gb)$:
\begin{enumerate}
  \item $\omega(\ga,\gb;\gc,\gd)<1,\quad
\omega(\ga,\gc;\gd,\gb)=1,\quad\omega(\ga,\gd;\gc,\gb)< 1,$
  \item $\omega(\ga,\gb;\gc,\gd)>1,\quad\omega(\ga,\gc;\gd,\gb)<1,\quad\omega(\ga,\gd;\gc,\gb)=1,$
     \item $\omega(\ga,\gb;\gc,\gd)=1,\quad\omega(\ga,\gc;\gd,\gb)>1,\quad \omega(\ga,\gd;\gc,\gb)>1,$
    \item $\omega(\ga,\gb;\gc,\gd)=\omega(\ga,\gc;\gd,\gb)=\omega(\ga,\gd;\gc,\gb)=1.$
\end{enumerate}
\end{lem}
\begin{proof}
$X$ is a $0$-hyperbolic space and therefore 
 by \cite[Lemmas 2.1.4. and 2.2.2.]{buyaloschroeder}, for any $o\in X$ and any distinct $\ga,\gb,\gc,\gd\in\partial X$, the two smallest entries of the following triple are equal:
\begin{align*}
\left((\ga\mid\gb)_o+(\gc\mid\gd)_o,(\ga\mid\gc)_o+(\gb\mid\gd)_o,(\ga\mid\gd)_o+(\gb\mid\gc)_o\right)
\end{align*}
Equivalently, the two largest entries of 
\begin{align*}
(d_o(\ga,\gb)d_o(\gc,\gd),d_o(\ga,\gc)d_o(\gb,\gd),d_o(\ga,\gd)d_o(\gb,\gc))
\end{align*}
are equal.
There are therefore four cases:
\begin{enumerate}
\item $d_o(\ga,\gb)d_o(\gc,\gd)=d_o(\ga,\gd)d_o(\gb,\gc)>d_o(\ga,\gc)d_o(\gb,\gd)$
\item $d_o(\ga,\gb)d_o(\gc,\gd)=d_o(\ga,\gc)d_o(\gb,\gd)>d_o(\ga,\gd)d_o(\gb,\gc)$
\item $d_o(\ga,\gd)d_o(\gb,\gc) =d_o(\ga,\gc)d_o(\gb,\gd)> d_o(\ga,\gb)d_o(\gc,\gd)$
\item $d_o(\ga,\gb)d_o(\gc,\gd)=d_o(\ga,\gc)d_o(\gb,\gd)=d_o(\ga,\gd)d_o(\gb,\gc)$
\end{enumerate}
each corresponding to the four cases listed in the statement.
\end{proof}
 
The following lemma explains why the four cases outlined in the above lemma correspond to the four cases outlined in Figure~\ref{fig:trees}.

\begin{lem}\label{Lem:treeintersections}
Let $X$ be a tree and consider pairwise distinct $\ga,\gb,\gc,\gd\in\partial X$ and all six geodesics between them. There are only four possible cases for how these geodesics intersect and these correspond to the four cases outlined in Figure~\ref{fig:trees}.  
\end{lem}
\begin{proof}
Given any distinct $\ga,\gb,\gc,\gd\in\partial X$, we know by Lemma~\ref{Lem:fourcases}, that there are only four possible cases for the values of the cross ratios. By Proposition~\ref{Prop:crossratiometric},
\begin{align*}
\omega(\ga,\gc;\gd,\gb)=1\iff p(\ga,\gc;\gd)=p(\ga,\gc;\gb)
\end{align*}
which is equivalent to $p(\gc,\gd;\ga)=p(\ga,\gb;\gc)$, since for trees the point $p(\ga,\gb;\gc)$ does not depend on the permutation of $\ga,\gb,\gc\in\partial X$ (see Figure~\ref{fig:specialpoint}). 
By the definition of the cross ratio
\begin{align*}
\omega(\ga,\gc;\gd,\gb)=1\iff\omega(\gd,\gb;\ga,\gc)=1 
\end{align*}
which is equivalent to $p(\ga,\gb;\gd)=p(\gc,\gd;\gb).$
Similarly, 
\begin{align*}
\omega(\ga,\gd;\gc,\gb)= 1\iff p(\gc,\gd;\ga)= p(\ga,\gb;\gd)\iff p(\ga,\gb;\gc)=p(\gc,\gd;\gb)
\end{align*}
therefore we can rewrite the four cases of Lemma~\ref{Lem:fourcases} in terms of equalities and inequalities of the points $p(\ga,\gb;\gc)$, $p(\ga,\gb;\gd)$, $p(\gc,\gd;\ga)$ and $p(\gc,\gd;\gb)$:
\begin{enumerate}
  \item $p(\ga,\gb;\gc)=p(\gc,\gd;\ga)\neq p(\ga,\gb;\gd)=p(\gc,\gd;\gb)$  
  \item $p(\ga,\gb;\gd)=p(\gc,\gd;\ga)\neq p(\ga,\gb;\gc)=p(\gc,\gd;\gb)$
\item
$p(\ga,\gb;\gc)=p(\ga,\gb;\gd)\neq p(\gc,\gd;\ga)=p(\gc,\gd;\gb)$
\item 
$p(\ga,\gb;\gc)=p(\ga,\gb;\gd)=p(\gc,\gd;\ga)=p(\gc,\gd;\gb)$
\end{enumerate}
These four possible cases clearly correspond to the four cases pictured in Figure~\ref{fig:trees}. 
\end{proof}
In the next proposition, we give a uniform condition on trees and real hyperbolic space for when two geodesics intersect. Before this, we introduce the following notation:
\begin{align*}
\oplus:&=\begin{cases}
+&\text{if $X$ is a hyperbolic space}\\
\max&\text{if $X$ is a tree}
\end{cases}
\end{align*}
\begin{prop}\label{Prop:uniformintersects}
Let $X$ be either a tree or $\mathbb{R}\sH^n$, and consider pairwise distinct $\ga,\gb,\gc,\gd\in\partial X$. Then, $[\ga,\gb]$ intersects $[\gc,\gd]$ if and only if 
\begin{align*}
\omega(\ga,\gc;\gd,\gb)\oplus \omega(\ga,\gd;\gc,\gb)=1
\end{align*}
\end{prop}
\begin{proof}
For a real hyperbolic space, any pair of intersecting geodesics are contained in an isometrically embedded real hyperbolic plane (cf. Section~\ref{Section:IsometricEmbeddings}) and therefore, the statement follows directly from Theorem~\ref{Thm:hyperbolicintersects}.\\
For a tree, note that $[\ga,\gb]$ and $[\gc,\gd]$ intersect in cases 1,2 and 4 of Figure~\ref{fig:trees} and by Lemma~\ref{Lem:treeintersections}, these correspond to cases 1,2 and 4 of Lemma~\ref{Lem:fourcases}. Indeed, in these cases
\begin{align*}
\max\{\omega(\ga,\gc;\gd,\gb),\omega(\ga,\gd;\gc,\gb)\}=1
\end{align*}
whereas in case 3, $\max\{\omega(\ga,\gc;\gd,\gb),\omega(\ga,\gd;\gc,\gb)\}>1.$
\end{proof}
It follows from Theorem~\ref{Thm:hyperbolicintersects} that if $\ga,\gb,\gc,\gd\in\partial X$ are any four distinct boundary points of a hyperbolic space (not necessarily real) and the above condition on the cross ratios holds, then $[\ga,\gb]$ and $[\gc,\gd]$ intersect. However, the condition will not hold for a pair of intersecting geodesics which are not contained in an isometric embedding of the real hyperbolic plane. If $\ga,\gb,\gc,\gd\in\partial X$ are any four distinct boundary points of a tree, then the above proposition tell us exactly when $[\ga,\gb]$ and $[\gc,\gd]$ are two intersecting geodesics. However, this does not account for geodesic lines of the form $[\ga,\gb]$ and $[\ga,\gc]$, which intersect in every point along the geodesic ray $[p(\ga,\gb;\gc),\ga]$. The following proposition shows that any pair of intersecting geodesics in a tree or hyperbolic space satisfy a condition in terms of cross ratios.
\begin{figure}
     \centering
\includegraphics[width=0.9\linewidth]{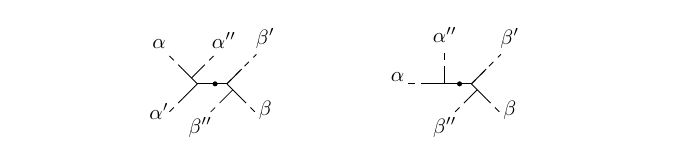}
 \caption{If $[\ga,\gb]$ and $[\ga',\gb']$ are two intersecting geodesics of a tree and $o\in[\ga,\gb]\cap[\ga',\gb']$, then one may always find a third geodesic $[\ga'',\gb'']$ such that $o\in[\ga'',\gb'']$. In both pictures $o$ is marked and on the left picture all boundary points are distinct whereas on the right $\ga=\ga'$.}
 \label{fig:treesthirdgeodesic}
 \end{figure}
\begin{prop}\label{Prop:thirduniformintersects}
Let $X$ be either a tree or a hyperbolic space and let $\ga,\gb,\ga',\gb'\in\partial X$. If
$[\ga,\gb]$ and $[\ga',\gb']$ are two intersecting geodesics, i.e. $[\ga,\gb]_{\gc}(t)=[\ga',\gb']_{\gc'}(s)$ for some $\gc,\gc'\in\partial X$ and some $s,t\in\mathbb{R}$, then there exist $\ga'',\gb'',\gc''\in\partial X$ such that
\begin{align*}
\omega(\ga,\ga'';\gb'',\gb)\oplus\omega(\ga,\gb'';\ga'',\gb)=1\qquad
\omega(\ga',\ga'';\gb'',\gb')\oplus\omega(\ga',\gb'';\ga'',\gb')=1
\end{align*}
and $[\ga,\gb]_{\gc}(t)=[\ga'',\gb'']_{\gc''}(r)=[\ga',\gb']_{\gc'}(s)$ for some $r\in\mathbb{R}$.
\end{prop}
\begin{proof}
Suppose $X$ is a hyperbolic space and $[\ga,\gb]$ and $[\ga',\gb']$ intersect in a point $o$, then $\ga,\gb,\ga',\gb'$ must be distinct. As in Section~\ref{Section:IsometricEmbeddings}, we may suppose that $o=[x]$ where $x=(0,\ldots,0,1)$ and let $u,v\in x^{\perp}$ be the distinct unit tangent vectors corresponding to $[\ga,\gb]$ and $[\ga',\gb']$. In \cite[Section 3.4.]{bourdon96} Bourdon shows that there exists $w\in x^{\perp}$ such that $\langle u\mid w\rangle\in\mathbb{R}$ and $\langle v\mid w\rangle\in\mathbb{R}$, thus proving that there exists a geodesic line $[\ga'',\gb'']$ through $o$ such that $[\ga,\gb]$ and $[\ga'',\gb'']$ are contained in an isometric embedding of the real hyperbolic plane and $[\ga',\gb']$ and $[\ga'',\gb'']$ are also contained in an isometric embedding of the real hyperbolic plane. Since $\ga,\gb,\ga',\gb',\ga'',\gb''$ are distinct, the conditions on the cross ratio hold by Proposition~\ref{Prop:uniformintersects}. Since $[\ga'',\gb'']$ contains $o$, we can parametrise this geodesic (cf. Definition~\ref{Def:parametrisedgeodesicnotation}) using any $\gc''\in\partial X\setminus\{\ga'',\gb''\}$ and we can find $r\in\mathbb{R}$ such that $o=[\ga'',\gb'']_{\gc''}(r)$.

Suppose $X$ is a tree and $[\ga,\gb]$ and $[\ga',\gb']$ intersect, in this case $\ga,\gb,\ga',\gb'$ are not necessarily distinct. However, since $X$ is thick, there are infinitely many points of $[\ga,\gb]$ of the form $p(\ga,\gb;\gd)$ for some $\gd\in\partial X\setminus\{\ga,\gb\}$, so for any $o\in[\ga,\gb]\cap[\ga',\gb']$ we may find $\ga'',\gb''\in\partial X\setminus\{\ga,\gb,\ga',\gb'\}$ such that 
$o\in[p(\ga,\gb;\ga''),p(\ga,\gb;\gb'')]$ (see Figure~\ref{fig:treesthirdgeodesic}).
 By construction, $o\in[\ga'',\gb'']$ and we may find $\gc''\in\partial X\setminus\{\ga'',\gb''\}$ and $r\in\mathbb{R}$ such that $o=[\ga'',\gb'']_{\gc''}(r)$. Since $\ga,\gb,\ga'',\gb''$ are  distinct and $\ga',\gb',\ga'',\gb''$ are also distinct and $[\ga'',\gb'']$ intersects $[\ga,\gb]$ and $[\ga',\gb']$ it follows from Proposition~\ref{Prop:uniformintersects} that the conditions on the cross ratio hold.
\end{proof}

If $[\ga,\gb]$ and $[\gc,\gd]$ intersect in a tree, then their intersection is the geodesic segment $[p(\ga,\gb;\gc),p(\ga,\gb;\gd)]$, below is the analogue of this property for the hyperbolic spaces.

\begin{prop}\label{Prop:symmetricintersects}
Let $X$ be a hyperbolic space. Suppose $[\ga,\gb]$ and $[\gc,\gd]$ intersect at $o\in X$ for some $\ga,\gb,\gc,\gd\in\partial X$. Then $o$ is the midpoint of $p(\ga,\gb;\gc)$ and $p(\ga,\gb;\gd)$.
\end{prop}
\begin{proof}
 Since $X$ is a symmetric space, we know there exists an isometry $\varphi_o$ of $X$ which fixes $o$ and reverses the direction of all geodesics through $o$. 

Let $x(t)$ be the geodesic parametrisation of $[\ga,\gb]$  with $x(-\infty)=\ga,\,x(0)=o,\,x(\infty)=\gb$ and let $y(t)$ be the geodesic parametrisation of $[\gc,\gd]$ with $y(-\infty)=\gc,\,y(0)=o,\,y(\infty)=\gd$. Then, using the definition of the Gromov product (see Section~\ref{Section:BourdonMetric}),
\begin{align*}
(\gb\mid\gd)_p=\lim_{t\rightarrow\infty}(x(t)\mid y(t))_p,\qquad
(\ga\mid\gc)_p=\lim_{t\rightarrow\infty}(x(-t)\mid y(-t))_p.
\end{align*}
for all $p\in X$.
Since the Gromov product $(x(t)\mid y(t))_p$ depends only on the metric and $\varphi_o$ is an isometry, we know that
\begin{align*}
(x(t)\mid y(t))_p=(\varphi_o(x(t))\mid \varphi_o(y(t)))_{\varphi_o(p)}=(x(-t)\mid y(-t))_{\varphi_o(p)}
\end{align*}
for all $t\in\mathbb{R}$.
It follows that $(\ga\mid \gc)_p=(\gb\mid\gd)_{\varphi_o(p)}$ and $(\gb\mid \gc)_p=(\ga\mid\gd)_{\varphi_o(p)}$. 
Therefore, the image of $p(\ga,\gb;\gc)$ under $\varphi_o$ is $p(\ga,\gb;\gd)$. Since $\varphi_o$ is an isometry fixing $o$, it follows that $d(o,p(\ga,\gb;\gc))=d(o,p(\ga,\gb;\gd))$.
\end{proof}

\begin{figure}
    \centering
\includegraphics[width=0.8\linewidth]{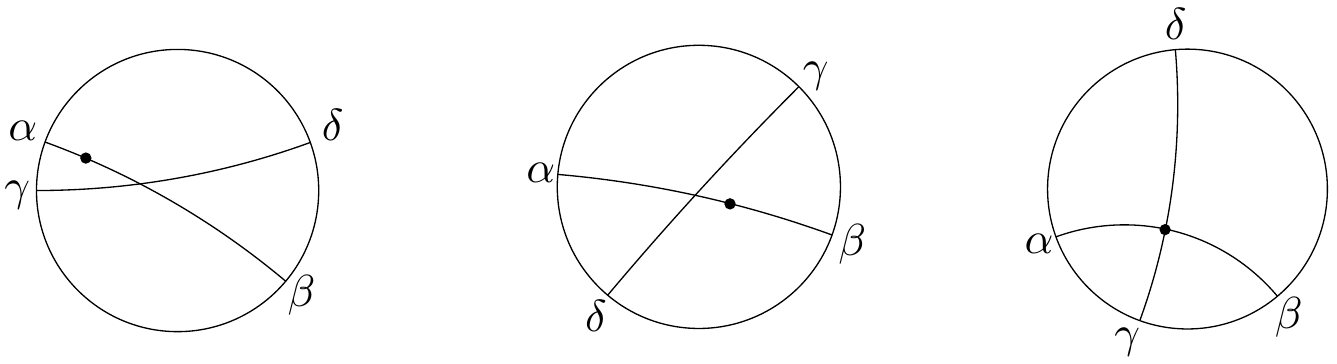}
\caption{Three cases for intersecting geodesics in $\mathbb{R}\sH^2$. From left to right: $\omega(\ga,\gb;\gc,\gd)<1$, $\omega(\ga,\gb;\gc,\gd)>1$, $\omega(\ga,\gb;\gc,\gd)=1$. In each case the point $p(\ga,\gb;\gc)$ is marked. }
\label{Figure:hyperbolicplanes}
\end{figure}

Now we are ready to give a uniform description of the intersection of $[\ga,\gb]$ and $[\gc,\gd]$. Before this, we introduce the following notation: 
\begin{align*}
\chi(\ga,\gb;\gc,\gd):&=
\begin{cases}
\left\{\dfrac{\abs{\ln\omega(\ga,\gb;\gc,\gd)}}{2}\right\}
&\text{if $X$ is a hyperbolic space}\\
[0,\abs{\ln\omega(\ga,\gb;\gc,\gd)}]
&\text{if $X$ is a tree}
\end{cases}
\end{align*}

\begin{prop}\label{Prop:uniformintersection}
Let $X$ be either a hyperbolic space or a tree. If pairwise distinct $\ga,\gb,\gc,\gd\in\partial X$ satisfy 
\begin{align*}
&\omega(\ga,\gc;\gd,\gb)\oplus\omega(\ga,\gd;\gc,\gb)=1
&\omega(\ga,\gb;\gc,\gd)\leq 1
\end{align*}
then
$[\ga,\gb]_{\gc}(s)=[\gc,\gd]_{\ga}(t)$ if and only if $s=t\in\chi(\ga,\gb;\gc,\gd)$.
\end{prop}
\begin{proof}
By Theorem~\ref{Thm:hyperbolicintersects} and Proposition~\ref{Prop:uniformintersects}, we know that $[\ga,\gb]$ and $[\gc,\gd]$ intersect and by  Proposition~\ref{Prop:crossratiometric} we know that 
\begin{align*}
d(p(\ga,\gb;\gc),p(\ga,\gb;\gd))=\abs{\ln\omega(\ga,\gb;\gc,\gd)}= \abs{\ln\omega(\gc,\gd;\ga,\gb)}=d(p(\gc,\gd;\ga),p(\gc,\gd;\gb)) 
\end{align*}
and $p(\ga,\gb;\gc)\in[p(\ga,\gb;\gd),\ga],\,p(\gc,\gd;\ga)\in[p(\gc,\gd;\gb),\gc]$.
 It follows that  
\begin{align*}
p(\ga,\gb;\gd)=[\ga,\gb]_{\gc}(t),\qquad p(\gc,\gd;\gb)=[\gc,\gd]_{\ga}(t),
\end{align*} 
for $t=\abs{\ln\omega(\ga,\gb;\gc,\gd)}$.

If $X$ is a hyperbolic space then the intersection point of $[\ga,\gb]$ and $[\gc,\gd]$ is the midpoint of $p(\ga,\gb;\gc)$ and $p(\ga,\gb;\gd)$, equivalently the midpoint of $p(\gc,\gd;\ga)$ and $p(\gc,\gd;\gb)$ (see the proof of Proposition~\ref{Prop:symmetricintersects}) and therefore
\begin{align*}
[\ga,\gb]_{\gc}(s)=[\gc,\gd]_{\ga}(t)\iff 
s=t=\dfrac{\abs{\ln\omega(\ga,\gb;\gc,\gd)}}{2}.
\end{align*}

If $X$ is a tree then the intersection of $[\ga,\gb]$ and $[\gc,\gd]$ is $[p(\ga,\gb;\gc), p(\ga,\gb;\gd)]$, equivalently the geodesic segment $[p(\gc,\gd;\ga),p(\gc,\gd;\gb)]$ (see the proof of Lemma~\ref{Lem:treeintersections}). It follows that 
\begin{align*}
[\ga,\gb]_{\gc}(s)=[\gc,\gd]_{\ga}(t)
\iff s=t\in[0,\abs{\ln\omega(\ga,\gb;\gc,\gd)}].
\end{align*}
\end{proof}

Proposition~\ref{Prop:thirduniformintersects} tells us exactly when two geodesics in a tree or hyperbolic space intersect, in terms of the cross ratio on the boundary. Combining this with the above proposition, we also understand how two geodesics in a tree or hyperbolic space intersect, i.e. we may give their intersection points using the boundary points and cross ratio. We will now use this to define a model space for $X$.
\begin{defin}\label{Def:equivalencerelation}
Let $X$ be either a hyperbolic space or a tree and define
\begin{align*}
\Omega(X):=\bigslant{\bigcup\limits_{\ga,\gb\in\partial X}\chi_{\ga\gb}}{\sim},
\end{align*}
where $\sim$ is the equivalence relation generated by the following conditions:
\begin{enumerate}
\item For all $t\in\mathbb{R}$ and all $\ga,\gb,\gc\in\partial X$: $(\ga,\gb,\gc,t)\sim(\gb,\ga,\gc,-t)$.
\item If $\ga,\gb,\gc,\gd\in\partial X$ satisfy $\omega(\ga,\gc;\gd,\gb)\oplus\omega(\ga,\gd;\gc,\gb)=1, \,\omega(\ga,\gb;\gc,\gd)\leq 1$, then 
$(\ga,\gb,\gc,s)\sim(\gc,\gd,\ga,t)$ if and only if   $s=t\in\chi(\ga,\gb;\gc,\gd)$.
\end{enumerate}
\end{defin}
We conclude by giving the proof of the main theorem. 
\begin{thm:maintheorem}
Let $(X,d_X)$ be either a hyperbolic space or a tree. 
Then, the map
\begin{align*}
&d_{\omega}:\Omega(X)\times\Omega(X)\rightarrow\mathbb{R}\qquad
&d_{\omega}((\ga,\gb,\gc,s),(\ga,\gb,\gc,t)):=\abs{s-t}
\end{align*}
is a well-defined metric on $\Omega(X)$ and 
\begin{align*}
\varphi:\Omega(X)\rightarrow X\qquad
\varphi:(\ga,\gb,\gc,t)\mapsto[\ga,\gb]_{\gc}(t)
\end{align*}
is an isometry.
\end{thm:maintheorem}
\begin{proof}
We start by showing that $\varphi$ is a well-defined bijection. In order to show that $\varphi$ is well-defined, we need to verify that if $(\ga,\gb,\gc,t)=(\gb,\ga,\gc,-t)$ then $[\ga,\gb]_{\gc}(t)=[\gb,\ga]_{\gc}(-t)$, which is clearly true for all $t\in\mathbb{R}$ and all $\ga,\gb,\gc\in\partial X$. We further need to verify that if $(\ga,\gb,\gc,s)=(\gc,\gd,\ga,t)$ for some $\ga,\gb,\gc,\gd\in\partial X$ and some $s,t\in\mathbb{R}$, then $[\ga,\gb]_{\gc}(s)=[\gc,\gd]_{\ga}(t)$, this follows from Proposition~\ref{Prop:uniformintersection}. So we know that $\varphi$ is well-defined and since every point of $X$ lies in at least one geodesic $[\ga,\gb]$, we certainly have surjectivity - it is left to prove injectivity.

Suppose that  $[\ga,\gb]_{\gc}(s)=[\ga',\gb']_{\gc'}(t)$ for some $s,t\in\mathbb{R}$ and  $\ga,\gb,\gc,\ga',\gb',\gc'\in \partial X$ then by Proposition~\ref{Prop:thirduniformintersects}, there exist
$\ga'',\gb'',\gc''\in\partial X$ such that
\begin{align*}
\omega(\ga,\ga'';\gb'',\gb)\oplus\omega(\ga,\gb'';\ga'',\gb)=1\qquad
\omega(\ga',\ga'';\gb'',\gb')\oplus\omega(\ga',\gb'';\ga'',\gb')=1
\end{align*}
and 
\begin{align*}
[\ga,\gb]_{\gc}(s)=[\ga'',\gb'']_{\gc''}(r)=[\ga',\gb']_{\gc'}(t)
\end{align*} 
for some $r\in\mathbb{R}$.
Furthermore, by 
 Lemma~\ref{Lem:parametriseddistance}, there exist $s',t',r',r''\in\mathbb{R}$ such that
\begin{align*}
&[\ga,\gb]_{\gc}(s)= 
[\ga,\gb]_{\ga''}(s'),\qquad [\ga',\gb']_{\gc'}(t)=[\ga',\gb']_{\ga''}(t),\\ &\qquad[\ga'',\gb'']_{\gc''}(r)=[\ga'',\gb'']_{\ga}(r')=[\ga'',\gb'']_{\ga'}(r'').
\end{align*}
By Proposition~\ref{Prop:geodesicmetric}, this is equivalent to
\begin{align*}
&(\ga,\gb,\gc,s)= 
(\ga,\gb,\ga'',s'),\qquad (\ga',\gb',\gc',t)=(\ga',\gb',\ga'',t),\\ &\qquad(\ga'',\gb'',\gc'',r)=(\ga'',\gb'',\ga,r')=(\ga'',\gb'',\ga',r'').
\end{align*}
It follows that 
\begin{align*}
&[\ga,\gb]_{\gc}(s)=[\ga'',\gb'']_{\gc''}(r)=[\ga',\gb']_{\gc'}(t)\\ &\qquad \iff [\ga,\gb]_{\ga''}(s')=[\ga'',\gb'']_{\ga}(r')=[\ga'',\gb'']_{\ga'}(r'')=[\ga',\gb']_{\ga''}(t)
\end{align*}
and 
\begin{align*}
&(\ga,\gb,\gc,s)= (\ga'',\gb'',\gc'',r)=(\ga',\gb',\gc',t)\\&\qquad\iff (\ga,\gb,\ga'',s')=(\ga'',\gb'',\ga,r')=(\ga'',\gb'',\ga',r'')=(\ga',\gb',\ga'',t).
\end{align*}
Clearly, the right hand side of both of these equivalances are equivalent to each other by the definition of $\Omega(X)$ and Proposition~\ref{Prop:uniformintersection}. It follows that $\varphi$ is a bijection.

Given any two points of $X$ we may find a geodesic $[\ga,\gb]$ containing them. We may therefore express any two points of $X$ as $[\ga,\gb]_{\gc}(s)$ and $[\ga,\gb]_{\gc}(t)$ for some $\ga,\gb,\gc\in\partial X$ and some $s,t\in\mathbb{R}$. Since $\varphi$ is bijective, we may therefore express any two elements of $\Omega(X)$ as $(\ga,\gb,\gc,s),(\ga,\gb,\gc,t)$ so $d_{\omega}$ is well-defined. Since $[\ga,\gb]_{\gc}$ is a geodesic we know that
\begin{align*}
d_X([\ga,\gb]_{\gc}(s),[\ga,\gb]_{\gc}(t))=\abs{s-t}
\end{align*} and 
therefore
\begin{align*}
d_{\omega}((\ga,\gb,\gc,s),(\ga,\gb,\gc,t))=d_X(\varphi(\ga,\gb,\gc,s),\varphi(\ga,\gb,\gc,t))
\end{align*}
which shows that
$d_{\omega}$ is a metric on $\Omega(X)$ and $\varphi$ is an isometry.
\end{proof}
\section*{Acknowledgements} 
I am grateful for the support received by the DFG priority programme SPP 2026 \textit{Geometry at Infinity}. 
I would like to thank Petra Schwer and Linus Kramer for suggesting this project and for helpful discussions, I would also like to thank Daniel Allcock for helping me understand the hyperbolic plane over the octonions. Some of this work was carried out in Münster, during an enjoyable and productive stay in the winter of 2022/2023, I am very grateful to Linus Kramer and Mathematics Münster for their hospitality. 
\bibliography{main}
\bibliographystyle{amsplain} 
\end{document}